\documentclass[a4paper,12pt,reqno]{amsart}
\usepackage{amscd}
\usepackage{amsmath}
\usepackage{ulem}
\usepackage{amssymb}
\usepackage{amsthm}
\usepackage[usenames,dvipsnames]{color}
\usepackage{color,soul}
\usepackage{enumerate, mdwlist}  
\usepackage{tikz}
\usepackage{url}

\input xy
\xyoption{all}  

\begin{document}

\newtheorem{thm}{Theorem}
\newtheorem{prop}[thm]{Proposition}
\newtheorem{conj}[thm]{Conjecture}
\newtheorem{lem}[thm]{Lemma}
\newtheorem{cor}[thm]{Corollary}
\newtheorem{axiom}[thm]{Axiom}
\newtheorem{sheep}[thm]{Corollary}
\newtheorem{deff}[thm]{Definition}
\newtheorem{fact}[thm]{Fact}
\newtheorem{example}[thm]{Example}
\newtheorem{slogan}[thm]{Slogan}
\newtheorem{remark}[thm]{Remark}
\newtheorem{quest}[thm]{Question}
\newtheorem{zample}[thm]{Example}

\newcommand{\sthat}{\hspace{.1cm}| \hspace{.1cm}}
\newcommand{\id}{\operatorname{id} }
\newcommand{\acl}{\operatorname{acl}}
\newcommand{\dcl}{\operatorname{dcl}}
\newcommand{\irr}{\operatorname{irr}}
\newcommand{\aut}{\operatorname{Aut}}
\newcommand{\fix}{\operatorname{Fix}}

\newcommand{\oo}{\mathcal{O}}
\newcommand{\aaa}{\mathcal{A}}
\newcommand{\mm}{\mathcal{M}}
\newcommand{\curg}{\mathcal{G}}
\newcommand{\bbf}{\mathbb{F}}
\newcommand{\A}{\mathbb{A}}
\newcommand{\R}{\mathbb{R}}
\newcommand{\Q}{\mathbb{Q}}
\newcommand{\C}{\mathbb{C}}
\newcommand{\cc}{\mathcal{C}}
\newcommand{\dd}{\mathcal{D}}
\newcommand{\N}{\mathbb{N}}
\newcommand{\Z}{\mathbb{Z}}
\newcommand{\cF}{\mathcal F}
\newcommand{\cB}{\mathcal B}
\newcommand{\cU}{\mathcal U}
\newcommand{\cV}{\mathcal V}
\newcommand{\cG}{\mathcal G}
\newcommand{\cD}{\mathcal D}
\newcommand{\curly}{\mathcal{C}}
\newcommand{\durly}{\mathcal{D}}
\newcommand{\fff}{\mathcal{F}}
\newcommand{\calc}{\mathcal{C}}
\newcommand{\GG}{\mathbb{G}}
\newcommand{\PP}{\mathbb{P}}
\newcommand{\Gal}{\mathrm{Gal}}
\newcommand{\Aut}{\mathrm{Aut}}
\newcommand{\signature}{\mathrm{sign}}

\definecolor{mypink3}{cmyk}{0, 0.7808, 0.4429, 0.1412}
\newcommand{\red}[1]{{\color{red}{#1}}}

\newcommand{\mahrad}[1]{{\color{blue} \sf $\clubsuit\clubsuit\clubsuit$ Mahrad: [#1]}}
\newcommand{\ramin}[1]{{\color{red}\sf $\clubsuit\clubsuit\clubsuit$ Ramin: [#1]}}
\newcommand{\ramintak}[1]{{\color{mypink3}\sf $\clubsuit\clubsuit\clubsuit$ Ramin2: [#1]}}
\newcommand{\ramintakloo}[1]{{\color{green} \sf $\clubsuit\clubsuit\clubsuit$ Ramin3 [#1]}}
\newcommand{\gautam}[1]{\noindent {\red{\sf $\ddagger\ddagger\ddagger$
      Gautam: [#1]}}}

\DeclareRobustCommand{\hlgreen}[1]{{\sethlcolor{green}\hl{#1}}}
\DeclareRobustCommand{\hlcyan}[1]{{\sethlcolor{cyan}\hl{#1}}}
\newcommand{\Fmodtor}{F^\times / \mu(F)}
\title[Subrings]{Counting subrings of $\mathbb{Z}^n$ of non-zero co-rank }
\author{Sarthak Chimni} 
\address{Department of Mathematics, Statistics, and Computer Science, University of Illinois at Chicago, 851 S Morgan St (M/C 249), Chicago, IL 60607}
\email{sarthakchimni@gmail.com}
\author{Gautam Chinta}
\email{gchinta@ccny.cuny.edu}
\author{Ramin Takloo-Bighash}
\email{rtakloo@uic.edu}
\begin{abstract}
In this paper we study subrings of $\Z^{n+k}$ of co-rank $k$  We
relate the number of such subrings $R$ with torsion subgroup
$(\Z^{n+k}/R)_{\rm{tor}}$ of size $r$ to the number of full rank
subrings of $\Z^n$ of index $r$.
\end{abstract}
\maketitle

\section{Introduction} 
Let $\Z^n$ be the set of $n$-tuples $(x_1, \dots, x_n)$ of integers.  This set comes with a natural addition and multiplication given by 
$$
(x_1, \dots, x_n) + (y_1, \dots, y_n) = (x_1 + y_1, \dots, x_n + y_n), 
$$
and 
$$
(x_1, \dots, x_n) \cdot (y_1, \dots, y_n) = (x_1  \cdot  y_1, \dots, x_n  \cdot  y_n). 
$$
Under these operations $\Z^n$ is a ring.  As is well known the ring
$\Z^n$ has a simple additive group structure, but when it comes to its
multiplicative structure there are some very easy-to-state basic
questions that we do not know how to answer. For example, let $f_n(r)$
be the number of subrings $R$ of $\Z^n$ with identity of index
$r$. Necessarily then, $R$ is a free $\Z$-module of rank $n$. The
counting function $f_n(r)$ and associated generating series
$F_n(s) : = \sum_{r=1}^\infty f_n(r) r^{-s}$ are basic objects of
interest.  

The general theory developed by Grunewald, Segal and Smith \cite{GSS}
shows that $F_n(s)$ can be expressed as as Euler product of rational
functions of $p^{-s}$ over all primes $p$, but only for $n\leq 4$ has
this rational function been computed explicitly.  For $n=2$ this
expression is immediate.  It is originally due to Datskovksy and
Wright \cite{DW} for $n=3$ and Nakagawa \cite{N} for $n=4.$ In fact,
these authors studied the more general problem understanding the
distribution of orders in cubic or quartic algebras, a particular case
of which was the computation of the generating series $F_3(s)$ in
\cite{DW} and $F_4(s)$ in \cite{N}.  Liu \cite{Liu} proved a number of
interesting theorems about $f_n(r)$, including the computation of
$F_n(s) : = \sum_{r=1}^\infty f_n(r) r^{-s}$ for $n \leq 4$ by an
alternative method.

For $n > 4$ the situation is considerably more complicated.  Kaplan,
Marcinek, and Takloo-Bighash \cite{KMT}, by using the methods of
$p$-adic integration, obtained results for the location and order of
the rightmost pole of $F_5(s)$ without explicitly computing the
series. They also obtained estimates for the location of the first
pole of $F_n(s)$ for $n > 5$. One of the reasons to study the analytic
properties of the generating series $F_n(s)$ is to find asymptotic
formulae for $N_n(B) = \sum_{r \leq B} f_n(r)$. The theory of $p$-adic
integration \cite{GSS} shows that $N_n(B)$ grows like a non-zero
constant $C_n$ multiplied by $B^{\alpha(n)} (\log B)^{b(n) - 1}$ for
$\alpha(n) \in \Q$ and $b(n) \in \N$. Combining the results of
\cite{DW, N, KMT} we know the following
about the behavior of $N_n(B)$:
\begin{thm}
If $n \leq 5$ there is a constant $C_n$ such that 
$$
N_n(B) \sim  C_n B (\log B)^{{n \choose 2}-1}
$$
as $B \to \infty$. If $n \geq 6$, for any $\epsilon > 0$ we have 
$$
B (\log B)^{{n \choose 2}-1} \ll N_n(B) \ll_\epsilon B^{\frac{n}{2}-\frac{7}{6} + \epsilon}.
$$
\end{thm}
In fact, results of Brakenhoff \cite{Brakenhoff} and
Atanasov-Kaplan-Krakoff-Menzel \cite{akkm} give slightly better bounds
for $n\geq 6.$

As mentioned above $f_n(r)$ counts full rank $\Z$-submodules of $\Z^n$
that are of index $r$. A natural question to ask is whether
one can quantify the distribution of subrings of $\Z^n$ which as
$\Z$-submodules are not of rank $n$. Let us make this precise. Let
$\phi_n(r)$ be the number of full-rank sublattices of $\Z^n$ which are closed
under the multiplication of $\Z^n$.  It's a well-known fact (e.g., Proposition 2.3 of \cite{Liu}) that for each $n \geq 2, r \geq 1$ we have $f_n(r) = \phi_{n-1}(r)$. It turns out that for many purposes the function $\phi_n(r)$ is a more convenient function to work with---and in fact the theory developed in \cite{GSS} deals with the function $\phi_n(r)$.  

\

We now define an analogue of the function $\phi_n(r)$ for lattices of non-zero co-rank. For $0 \leq k \leq n$, define $\phi_{n, k}(r)$ be the number of sublattices $L$ of $\Z^n$ which have the following properties: 
\begin{itemize}
\item The lattice $L$ is closed under multiplication; 
\item as a $\Z$-submodule, $L$ is of co-rank $k$ in $\Z^n$; 
\item the size of the torsion subgroup of $\Z^n/ L$ is equal to $r$. 
\end{itemize}
Clearly, $\phi_{n, 0}(r) = \phi_n(r)$.  It turns out that the function $\phi_{n, k}(r)$ and $\phi_n(r)$ have a simple relationship. The following theorem is our main result. 
\begin{thm}\label{main-theorem}
For all $n, k, r$ we have
$$
\phi_{n+k, k}(r) = \left\{ {n + k +1 \atop n+1}\right\} \cdot \phi_n(r). 
$$
Here, for natural numbers $u, v$, $\left\{ {u \atop v} \right\}$ is
the Stirling number of second kind defined as the number of partitions
of a set with $u$ elements into $v$ non-empty subsets.
\end{thm}

The main step in the proof of this theorem is a rigidity result (Theorem \ref{rigidity}) which determines exactly what types of lattices contribute to the counting function $\phi_{n+k, k}(r)$. The rest of the proof consists of a combinatorial argument counting these lattices.  For information on Stirling numbers of the second kind, see \cite{Bogart}, especially Ch. 2, \S 3. 

\

The rigidity result mentioned above is the statement that matrices corresponding to multiplicative sublattices will be of very special shape.  The upshot of this result is that multiplicative sublattices of non-zero co-rank in $\Z^n$ are all obtained from full rank multiplicative sublattices in various $\Z^m$'s for $m < n$ in very specific ways.  Let us illustrate the results we are about to prove using co-rank two multiplicative sublattices in $\Z^4$. 

\

Define four maps $\Z^2 \to \Z^4$ by the following formulae: 
$$
f_1(x, y) = (x, y, 0, 0), 
$$
$$
f_2(x, y) = (x, y, y, 0), 
$$
$$
f_3(x, y) = (x, y, y, y), 
$$
$$
f_4(x, y) = (x, x, y, y). 
$$

We can make more maps $\Z^2 \to \Z^4$ by considering maps of the form $\tau \circ f_j \circ \sigma$ for $\sigma \in S_2, \tau \in S_4$---we call these maps {\em acceptable}.  For example, the map that sends $(x, y)$ to $(y, x , 0, x)$ is acceptable. A consequence of our rigidity result is that if $L$ is a multiplicative sublattice of co-rank two in $\Z^4$, then there is a multiplicative sublattice  $L'$ of full rank in $\Z^2$ such that $L = f(L')$ for some acceptable map $f$. Furthermore, the size of the torsion subgroup of $\Z^4/L$ is equal to the index of $L'$ in $\Z^2$. We will see that the scenario described here is completely general. 

\

Theorem \ref{main-theorem} was discovered thanks to the Online Encyclopedia of Integer Sequences (OEIS).  We computed a few values of the function $\sigma(n, k)$ by hand and then a search through OEIS revealed the connection to the Stirling Numbers of the Second Kind. These numbers appear under sequence A008277 in the Encyclopedia \cite{Online}.

\

The first named author is partially supported by NSF DMS 1601289.  The
second author wishes to thank the Simons Foundation for partial
support of his work through a Collaboration Grant. The authors also
wish to thank Nathan Kaplan for helpful conversations.

\

This paper is organized as follows. In \S \ref{section-rigidity} we review basic definitions and prove the rigidity theorem.  We prove  the main theorem in \S \ref{section-proof_main_thm}.

\section{Rigidity Theorem}\label{section-rigidity}
A {\em lattice} is a $\Z$-submodule of some $\Z^n$. When referring to a specific $\Z^n$ we usually speak of a {\em sublattice}. We call a sublattice $L$ of $\Z^n$ a {\em multiplicative sublattice} if for every $u, v \in L$ we have $u \cdot v \in L$. A multiplicative sublattice $L$ is a subring if it contains the identity element $(1, \dots, 1)$.  We refer the reader to Liu \cite{Liu} for basic properties of multiplicative lattices of full rank in $\Z^n$. 

\

Let $L$ be a lattice of rank $m$ in $\Z^n$. We define the {\em co-rank} of $L$ to be the integer $n - m$. The following lemma is an easy consequence of row operations. 

\begin{lem}\label{lower}
Given a lattice $L$ in $\Z^n$ of co-rank $k$ there is an $(n-k) \times n$ integral matrix $M= (x_{ij})$ such that $x_{ij} = 0$ whenever $j - i > k$, and with the property that the rows of $M$ generate $L$. \end{lem}

Note that the matrix $M$ as in the lemma is not unique. In fact, if $A$ is any $(n-k) \times (n-k)$  lower triangular integral matrix with determinant $1$, then $AM$ is another matrix that satisfies the conditions of the lemma.  

\

Let $M$ be the matrix corresponding to the lattice $L$ of co-rank $k$ as in Lemma \ref{lower}. Then $L$ is multiplicative if and only if for every two rows $v, w$ of $M$, $v \cdot w \in L$.  
\begin{prop}\label{corank-1}
Let L be a multiplicative sublattice of $\mathbb{Z}^n$ of co-rank 1.  Then L has a basis which forms the rows of a $ (n-1) \times n$ matrix M such that $M_{ij} = 0$ if $i<j-1$  and M has a column of zeros or two columns of M are identical.
\end{prop}

\begin{proof}
We prove this using induction on $n$. If $n =1$ then there is no sublattice of co-rank $1$ so the result is vacuously true. So we consider the case $n=2$. Any multiplicative sublattice $L$ of co-rank $1$ has rank 1 and therefore is generated by a non-zero row vector of length $2$,  
\[
M=
\begin{bmatrix}    
x_{11}       & x_{12} 
\end{bmatrix}.
\]
As $L$ is multiplicative, $M.M$ should be a scalar multiple of $M$. Hence we get the following equations:
\begin{subequations}
\begin{align}
x_{11}^2 = \lambda x_{11} \\
x_{12}^2 = \lambda x_{12}
\end{align}
\end{subequations}
Note that both $x_{11}$ and $x_{12}$ can't simultaneously be zero. If either of them are zero we get a zero column as desired and if both are non-zero we get that $x_{11} = \lambda = x_{12}$ and in that case both columns are identical. \\\\
Now we assume that the result holds for $n= k$ and show that it is true for $n=k+1$ 
Let $L$ be a multiplicative sublattice of $\mathbb{Z}^{k+1}$ of co-rank $1$. Then $L$ has a basis which forms the rows of a matrix $M = (x_{ij})$ such that $x_{ij} = 0$ for $i < j-1$.  Now $M$ can be written as 
\[
M=
\begin{bmatrix}    
M'       &     0 \\
v        &     x_{k,k+1} 
\end{bmatrix}.
\] 
If $x_{k,k+1} = 0$ then we have a column of zeros and we have nothing to prove. So from here on we assume that $x_{k,k+1} \neq 0$.
Clearly $M'$ represents a multiplicative sublattice of
$\mathbb{Z}^{k}$. By the induction hypothesis $M'$ has a column of zeros or a pair of identical columns. \\\\
\textbf{Case 1 : } $M'$ has a column of zeros.

\

Suppose the $j$th column of $M'$ is 0. If $x_{k,j} = 0$ we are done. So we assume that $x_{k,j} \neq 0$. Consider the product of the bottom row $R_{k}$ of $M$ with itself. Write \[R_{k}^2 =  \sum_{m=1}^{m=k} \lambda_{m}R_{m}.\] So we have the following equations.
\begin{subequations}
\begin{align}
x_{k,k+1}^2 = \lambda_{k} x_{k,k+1}\\
x_{k,j}^2 = \lambda_{k} x_{k,j}
\end{align}
\end{subequations}
As $x_{i,j} = 0$ for $ i \neq k$. Since both $x_{k,k+1}$ and $x_{k,j}$ are non-zero we have $x_{k,j} = \lambda_{k} = x_{k,k+1}$ which implies that the j{th} and $(k+1)$st columns are identical as all other entries are $0$. \\\\
\textbf{Case 2 : } $M'$ has a pair of identical columns.

\

Let the ith and jth columns of $M'$ be equal. We can assume that these are non-zero columns as the first case already deals with zero columns. Therefore there is $l<k$ such that $x_{l,i} = x_{l,j} \neq 0$. Now \[ R_{l} \cdot R_{k} = \sum_{m=1}^{m=k} \gamma_{m}R_{m}\] So we have
\begin{subequations}
\begin{align}
x_{k,i}x_{l,i} = \sum_{m=1}^{m=k} \gamma_{m}x_{m,i} \\
x_{k,j}x_{l,j} = \sum_{m=1}^{m=k} \gamma_{m}x_{m,j} \\
0 = \gamma_{k}x_{k,k+1}
\end{align}
\end{subequations}
$\gamma_{k} = 0$ as $x_{k,k+1} \neq 0$. This and the fact that $x_{m,i} = x_{m,j}$ for $m < k$ gives us that each term in the summations in (3a) and (3b) are equal which implies that the sums are equal. Therefore we have that in fact $x_{k,i}x_{l,i} =x_{k,j}x_{l,j}$. Since $x_{l,i} = x_{l,j} \neq 0$ we have $x_{k,i} = x_{k,j}$. So that the $i$th and $j$th columns of $M$ are identical.

\end{proof}
\begin{sheep}
Any basis of a multiplicative lattice L of co-rank 1 will form the rows of an $(n-1) \times n$ matrix M with $n-1$ distinct non-zero columns.
\end{sheep}
\begin{proof}
The property of having a column of zeros or two identical columns is invariant under elementary row operations. This means that any matrix whose rows are the basis of a multiplicative sublattice of co-rank $1$ of $\mathbb{Z}^n$ will have this property.
\end{proof}
\begin{thm}[Rigidity]\label{rigidity}
Let $L$ be a multiplicative sublattice of $\mathbb{Z}^{n}$ of co-rank k, then every basis of $L$  forms the rows of a $ (n-k) \times n$ matrix $M$ with exactly $n-k$ distinct non-zero columns.
\end{thm}

\begin{proof}
  The matrix $M$ has column rank $n-k$.  Let's say the first $n-k$
  columns are linearly independent, and hence, distinct and nonzero.
  Let $M^{(i)}$ be the $(n-k)\times (n-k+1)$-dimensional matrix
  obtained by appending the $i^{th}$ column of $M$ to the first $n-k$
  columns of $M$.  Since the rows $M^{(i)}$ generate a multiplicative
  sublattice of $\Z^{n-k+1}$ of corank 1, the previous corollary
  implies that $M^{(i)}$ has exactly $n-k$ distinct nonzero columns.
  Hence the $i^{th}$ column of $M$ must be equal to one of the first
  $n-k$ columns or 0.  Since this is true for all $i, n-k<i\leq n$, we
  conclude that the full matrix has exactly  $n-k$ distinct nonzero columns.
\end{proof}

\section{Proof of Theorem 2}\label{section-proof_main_thm}

We begin with a definition.
\begin{deff}
An injective map  
$$
g: \Z^n \to \Z^{n+k}
$$
of the form $g(x_1, \dots, x_n) = (y_1, y_2, \dots, y_{n+k})$ with
each $y_j$ either equal to some $x_i$ or $0$ is called {\em acceptable}.
\end{deff}

Theorem \ref{rigidity} can be formulated as follows: 

\begin{thm}\label{rigidity-2}
  Any multiplicative sublattice of co-rank $k$ in $\Z^{n+k}$ is of the
  form $g(L)$ where $ g : \Z^n \to \Z^{n+k}$ is an acceptable map and
  $L$ is a multiplicative sublattice of full rank in $\Z^n$.
\end{thm}

The next observation is simple but essential for what follows. 

\begin{lem}\label{index}
For any acceptable map $g$ and any sublattice $L$ in $\Z^n$ of rank $n$, we have 
$$
\# (\Z^{n+k} / g(L))_{\rm{tor}} = [\Z^n : L].
$$
\end{lem}

\begin{proof}
Let the Smith Normal form of the lattice L be D. Then $[\Z^n : L] = \text{Det(D)}$. Since in any basis the distinct columns of g(L) are the same as those of L, we have that the Smith Normal Form of L is $[D | 0]$. So that $\# (\Z^{n+k} / g(L))_{\rm{tor}} = [\Z^n : L]$.
\end{proof}

Two acceptable maps $g_1, g_2 : \Z^n \to \Z^{n+k}$ are called {\em
  equivalent} if there is a permutation $\tau \in S_n$ such that
$g_1 = g_2 \circ \tau$.  We next describe a complete set of
representatives for this equivalence relation.

Let $g: \Z^n \to \Z^{n+k}$ be an acceptable map and $\{f_i\}$ the standard
basis for $\Z^{n+k}$.  By definition of acceptable we can write
\[g(x_1,\ldots, x_n)=\sum_{i=1}^n x_i\left(\sum_{j\in A_i^g} f_j\right). 
\]
for a collection of subsets $\{A_1^g, \dots, A_n^g \}$ of
$\{1, \dots, n+k\}$. In fact, if we define
$A_0^g = \{0, \dots, n+k \} \setminus \bigcup_{i=1}^n A_i^g$, then
$\mathcal P^g: = \{A_0^g, \dots, A_n^g \}$ is a partition of
$\{0, \dots, n+k\}$.  We will call an acceptable function $g$ {\em
  ordered} if  $\min A_i^g < \min A_j^g$ whenever  $i < j$.  Given an
arbitrary acceptable map $g$ there exists exactly one permutation
$\tau \in S_n$ for which $g\circ \tau$ is ordered. That is, 

\begin{lem}
  The set of ordered acceptable maps is a set of representatives for
  the equivalence classes of acceptable maps $\Z^n \to \Z^{n+k}$ under
  the action of $S_n$.
\end{lem}

Going in the other direction, to a set partition $\mathcal P$ of
$\{0, \dots, n+k\}$ into $n+1$ parts, we may associate an ordered
acceptable map $g_{\mathcal P}$ as follows.  Begin by ordering
$\mathcal P=\{A_0,A_1,...,A_n\}$ in the following way: if $i<j$ then
$\min(A_i) < \min(A_j)$.  In particular, $0 \in A_0$. Define
\[g_{\mathcal P}(x_1,\ldots , x_n)=
\sum_{i=1}^n x_i\left(\sum_{j\in A_i} f_j\right). 
\]
For example $\{\{0,3,4\},\{1,5\},\{2,7,8\},\{6\}\}$ corresponds to the map from $\Z^3$ to $\Z^8$ which sends
$$(a,b,c) \mapsto (a,b,0,0,a,c,b,b),$$  i.e., $0$ is in the $3$ and $4$ spot, `$a$' in the $1$ and $5$ entries, `$b$' in 
the $2$,$7$ and $8$ entries and `$c$' in the $6$th entry.



It is clear that the maps $g\mapsto \mathcal P^g$ and
$\mathcal P\mapsto g_P$ are inverse to one another and provide a
bijection between ordered acceptable maps $\Z^n\to \Z^{n+k}$ and set
partitions of $\{0, \dots, n+k\}$ into $n+1$ parts.  This and the
definition of Stirling numbers of the second kind lead to the next corollary:

\begin{sheep}\label{coro}
The number of equivalence classes of acceptable maps $\Z^n \to \Z^{n+k}$ is equal to $\left\{ {n + k + 1 \atop n +1} \right\}$. 
\end{sheep}

The final step in the proof of Theorem \ref{main-theorem} is a
refinement of Theorem \ref{rigidity-2}.
\begin{prop}\label{acceptable-order}
  Any multiplicative sublattice of co-rank $k$ in $\Z^{n+k}$ is of the
  form $g(L)$ where $g : \Z^n \to Z^{n+k}$ is an ordered acceptable
  map and $L$ is a multiplicative sublattice of full rank in $\Z^n$ .
  Moreover, such $g$ and $L$ are uniquely determined.
\end{prop}
\begin{proof}
  A consequence of Theorem \ref{rigidity} is that any multiplicative
  sublattice $L'$ in $\Z^{n+k}$ will correspond to some partition
  $\mathcal P$ of $\{0,1,...n+k\}$ into $n+1$ parts. This partition is
  obtained by the same method that associated a partition to an
  acceptable map. The partition $\mathcal P$ corresponds to a unique
  ordered acceptable map $f_\mathcal P$. The lattice $L'$ is clearly
  in the range of $f_\mathcal P$ so ${f_\mathcal P}^{-1}(L')$ is the
  unique lattice $L$ in $\Z^n$ which maps to $L'$ under
  $f_\mathcal P$.\
\end{proof}

\begin{proof}[Proof of Theorem \ref{main-theorem}] Combine Proposition \ref{acceptable-order} with Corollary \ref{coro}.  \end{proof}

\end{document}